\documentclass{amsart}
\usepackage{amsfonts,amssymb,amsmath,amsthm}
\usepackage{url}
\usepackage{enumerate}
\usepackage{graphicx}
\usepackage{hyperref}
\usepackage{graphicx}
\usepackage{siunitx}
\usepackage{float}
\restylefloat{figure}
\newtheorem{theorem}{Theorem}[section]
\newtheorem{lemma}[theorem]{Lemma}

\author{Uuganbaatar Ninjbat}
\address{
Mathematics Department\\
The National University of Mongolia\\
Ulaanbaatar, Mongolia}
\email{uugnaa.ninjbat@gmail.com}
\thanks{To appear in Hitotsub. J. Econ. 59 (1).}

\keywords{The Gibbard-Satterthwaite theorem, Field extension lemma, Group contraction lemma}
\subjclass[2010]{91B12, 91B14}
\begin{document}
\title[A Missing Proof of The Gibbard-Satterthwaite Theorem]{A Missing Proof of The Gibbard-Satterthwaite Theorem}
\begin{abstract}
A short and direct proof of the Gibbard-Satterthwaite theorem \`{a} la Amartya Sen's proof of Arrow's impossibility theorem is given.
\end{abstract}
\maketitle

\section{Introduction}
\label{intro}
Several different approaches are used in proving social choice impossibility theorems, such as
\begin{itemize}
\item By induction (see e.g. \cite{satt}, \cite{sen01}, \cite{nin12-b} and \cite{sven}), 
\item By contradiction (see e.g. \cite{fish} and \cite{suz}), 
\item By the pivotal voter approach (for recent modifications, see \cite{yu}, \cite{fey}),
\item By other known impossibility theorems (see e.g. \cite{gibb}), and
\item By using computer assistance (see e.g.  \cite{tak}, \cite{tang-ling}).
\end{itemize}
Among these, proofs in the style of \cite{arr} and \cite{sen86} have a distinct feature to treat the impossibility result under consideration as a self contained mathematical structure and deduce the result from its setting without referring to an external mathematical device; see Sen's discussions in \cite{sen14}. In Sect. \ref{proof}, we give a short and direct proof for the Gibbard-Satterthwaite theorem (\cite{gibb}, \cite{satt}) using this approach.

\section{The preliminaries}
\label{setting} $A$ denotes the set of alternatives with $|A|\geq 3$
elements, and $X$ denotes the set of strict linear orders (i.e. complete, transitive and asymmetric binary relations) on $A$. Let there be $N$ individuals in the group $I=\{1,2,...,N\}$. 
A function $f:X^{N}\rightarrow A$ is called as a social choice function. A member $x=(x_{1},...,x_{N})$ of $X^{N}$ is called as a profile, and its $i^{\prime}$th component, $x_{i}$, is called the individual $i^{\prime}$s ranking. For any $x\in X^{N}$ and $i\in I$, let $(x_{i}^{\prime },x_{-i})\in X^{N}$ denote the profile
that has $x_{i}^{\prime }\in X$ in its $i^{\prime}$th component instead of $
x_{i}\in X,$ and otherwise the same as $x\in X^{N}$. When $a\in A$ is ranked above $b\in A$ according
to $x_{i}$ we write $a\succ _{x_{i}}b$. 

A group of individuals $G\subseteq I$ is \textit{decisive over} $a\in A$ if $f(x)=a$ for all $x\in X^{N}$ such that $a$ is on the top of $x_{i}$ for all $i\in G$. $G\subseteq I$ is \textit{decisive} if it is decisive over all $a\in A$. We say that $f:X^{N}\rightarrow A$ is \textit{unanimous} (UNM) if $I$ is decisive. It is 
\textit{manipulable} (MNP)\ at $x\in X^{N}$ by $i\in I$ via $x_{i}^{\prime}\in X$ if $f(x_{i}^{\prime },x_{-i})\succ _{x_{i}}f(x)$. It is \textit{
strategy proof} (STP) if it is not manipulable. Finally, it is \textit{dictatorial} (DT) if there is a decisive group consisting of a single individual. 
The following result is known as the Gibbard-Satterthwaite theorem:

\begin{theorem}[\cite{gibb}, \cite{satt}]
\label{GS} $f:X^{N}\rightarrow A$ is
UNM and STP if and only if it is DT.
\end{theorem}

\section{The proof}
\label{proof}
From now on we assume $f:X^{N}\rightarrow A$ is UNM and STP, and we shall prove three subsequent lemmas.
\begin{lemma}[Tops only]
\label{basic}
Let $x\in X^{N}$ and $a,b\in A$ be such that the top ranked alternative in $x_{i}$ is in $\{a,b\}$ for all $i\in I$. Then, $f(x)\in \{a, b\}$.
\end{lemma}
\begin{proof}
By UNM, we may assume that $a,b\in A$ are distinct. Let $G(a,x)=\{i\in I: x_{i} \text{ ranks $a$ as the top}\}$ and $G(b,x)=\{j\in I: x_{j} \text{ ranks $b$ as the top}\}$. Without losing generality, we may assume that $G(a,x)=\{1, ..., k\}$ and $G(b,x)=\{k+1, ..., N\}$ for some $k\leq N$.\footnote{The reader should check that our argument works if we instead had $G(a,x)=\{i_{1}, ...,i_{k}\}$ and $G(b,x)=\{i_{k+1}, ..., i_{N}\}$.} Consider $x'\in X^{N}$ such that $x'_{i}=(a\succ b\succ ...)$ for all $i\in G(a, x)$, and $x'_{j}=(b\succ a\succ ...)$ for all $j\in G(b,x)$. We claim that $f(x')\in\{a,b\}$. To see this, suppose $f(x')\notin \{a,b\}$. Transform $x'\in X^{N}$ by reversing the positions of $a,b\in A$ in $x'_{i}$ for all $i\in G(a,x)$, one at a time. Let $x^{i}$ be the resulting profile after $x'_{1}$, ..., $x'_{i}$ are changed, and we set $x^{0}=x'$. Then, $f(x^{i})\notin\{a,b\}$ for all $i\in G(a, x)$, since
\begin{itemize}
\item $f(x^{0})\notin\{a,b\}$, and
\item if $f(x^{i})\notin\{a,b\}$, but $f(x^{i+1})\in \{a,b\}$ for some $i\in \{0, ..., k-1\}$, then $f$ is MNP at $x^{i}\in X^{N}$ by individual $i+1$ via $x'_{i+1}$.
\end{itemize}
In particular, $f(x^{k})\notin\{a,b\}$, which then contradicts to UNM, and this proves our claim.

Transform $x'\in X^{N}$ by replacing $x'_{i}$ with $x_{i}$ for all $i\in G(a,x)$, one at a time. Let $y^{i}\in X^{N}$ be the profile obtained after $x'_{1}, ..., x'_{i}$ are replaced, and we set $y^{0}=x'$. Then, $f(y^{i})\in \{a,b\}$ for all $i\in G(a,x)$, since 
\begin{itemize}
\item $f(y^{0})\in\{a,b\}$, and 
\item if $f(y^{i})\in\{a,b\}$ then $f(y^{i+1})\in \{a,b\}$ for all $i\in \{0, ..., k-1\}$. To see this, suppose for some $i\in \{0, ..., k-1\}$, $f(y^{i})\in\{a,b\}$ but $f(y^{i+1})\notin \{a,b\}$. Start with $y^{i+1}\in X^{N}$ and transform everyone's preferences across $\{k+1, ..., N\}$ by bringing $a\in A$ to the top. Then STP ensures that social choice is never in $\{a,b\}$, which then eventually contradicts to UNM after transforming the individual $N'$s preferences.
\end{itemize}
In particular, $f(y^{k})\in \{a,b\}$. Now transform $x'\in X^{N}$ by replacing $x'_{j}$ with $x_{j}$ for all $j\in G(b,x)$, one at a time. Let $z^{j}\in X^{N}$ be the profile obtained after $x'_{k+1}, ..., x'_{j}$ are replaced, and we set $z^{k}=x'$. Then, $f(z^{j})\in \{a,b\}$ for all $j\in G(b,x)$, since 
\begin{itemize}
\item $f(z^{k})\in\{a,b\}$, and 
\item if $f(z^{j})\in\{a,b\}$ then $f(z^{j+1})\in \{a,b\}$ for all $j\in \{k, ..., N-1\}$. To see this suppose for some $j\in \{k, ..., N-1\}$, $f(z^{j})\in\{a,b\}$ but $f(z^{j+1})\notin \{a,b\}$. Start with $z^{j+1}\in X^{N}$ and transform everyone's preferences across $\{1, ..., k\}$ by bringing $b\in A$ to the top. Then STP ensures that social choice is never in $\{a,b\}$, which then eventually contradicts to UNM after transforming the individual  $k'$s preferences.
\end{itemize}
In particular, $f(z^{N})\in \{a,b\}$. We then claim that exactly one of the following two statements holds true: 
\begin{itemize}
\item[(a)] $f(y^{k})=b$, or
\item[(b)] $f(z^{N})=a$.
\end{itemize}
To see this, assume none of them holds true. Then, since both $f(y^{k})$ and $f(z^{N})$ are in $\{a,b\}$, we have $f(y^{k})=a$ and $f(z^{N})=b$. Transform $y^{k}\in X^{N}$ back to $x'\in X^{N}$ by reversing the above procedure. Then STP ensures that social choice remains at $a\in A$ throughout this transformation, in particular $f(x')=a$. Similarly, transform $z^{N}\in X^{N}$ back to $x'\in X^{N}$. Again, STP ensures that social choice remains at $b\in A$ throughout this transformation, in particular $f(x')=b$, which is a contradiction as we already concluded that $f(x')=a$ and $a, b\in A$ are different alternatives. Thus, at least one of the two statements must be true. However, with a very similar argument one can also show that the two statements in our claim can not be true at the same time. This proves our claim.

To complete the proof of Lemma \ref{basic}, assume $f(y^{k})=b$ and transform $y^{k}\in X^{N}$ into $x\in X^{N}$ by changing preferences of the individuals in $G(b,x)$ into their preferences in $x\in X^{N}$, one at a time. Then, STP ensures that social choice remains at $b\in A$ throughout this transformation. In particular, $f(x)=b$. If instead we had $f(z^{N})=a$, then we can show that $f(x)=a$ with a similar argument. Thus, in either case, $f(x)\in \{a,b\}$.
\end{proof}

\begin{lemma}[Extension]
\label{field} 
Let $G\subseteq I$ and $x\in X^{N}$ be such that $a\in A$ is ranked at the top of $x_{i}$ for all $i\in G$, and at the bottom of $x_{j}$ for all $j\in I\setminus G$. If $f(x)=a$, then G is decisive. 
\end{lemma}
\begin{proof}
As above, we may assume that $G=\{1,..., k\}$ with $k\leq N$. We first show that $G$ is decisive over $a\in A$. Let $x'\in X^{N}$ be such that $a\in A$ is at the top of $x'_{i}$ for $i\in G$. Transform $x\in X^{N}$ into $x'\in X^{N}$ by replacing $x_{i}$ with $x'_{i}$, for $i=1,2, ..., N$, one at a time. Let $x^{i}\in X^{N}$ be the profile obtained after changing $x_{1}, ..., x_{i}$, and we set $x^{0}=x$. Notice that $f(x^{i})=a$ for $i=1,..., k$, since
\begin{itemize}
\item $f(x^{0})=a$, and 
\item if $f(x^{i})=a$, but $f(x^{i+1})\neq a$ for some $i\in \{0, ..., k-1\}$, then $f$ is MNP by individual $i+1$ at $x^{i+1}$ via $x_{i+1}$. 
\end{itemize}
Notice also that $f(x^{j})=a$ for $j=k+1,..., N$, since 
\begin{itemize}
\item $f(x^{k})=a$, and 
\item if $f(x^{j})=a$, but $f(x^{j+1})\neq a$ for some $j\in \{k,..., N-1\}$, then $f$ is MNP by individual $j+1$ at $x^{j}$ via $x'_{j+1}$. 
\end{itemize}
We showed that $f(x')=f(x^{N})=a$ and hence, $G$ is decisive over $a$.

We next show that $G$ is decisive over any $b\in A\setminus \{a\}$. Let $y\in X^{N}$ be a profile where $y_{i}=(a\succ b\succ ...)$ for all $i\in G$, and $y_{j}=(c\succ ...\succ b)$ for all $j\in I\setminus G$, for some $c\in A\setminus\{a,b\}$. Then, $f(y)=a$ as $G$ is decisive over $a\in A$. Change the positions of $a,b\in A$ in $y_{i}$ for $i=1,2, ..., k$, one at a time. Let $y^{i}\in X^{N}$ be the profile obtained after changing $y_{1}, ..., y_{i}$, and we set $y^{0}=y$. Notice that $f(y^{i})\in \{a, b\}$ for $i=1,..., k$, since 
\begin{itemize}
\item $f(y^{0})=a\in \{a,b\}$, and 
\item if $f(y^{i})\in \{a,b\}$, but $f(y^{i+1})\notin \{a,b\}$ for some $i\in \{0, ..., k-1\}$, then $f$ is MNP by individual $i+1$ at $y^{i+1}\in X^{N}$ via $y_{i+1}$. 
\end{itemize}
Then, $f(y^{k})\in \{a, b\}$, but by Lemma \ref{basic}, $f(y^{k})\in \{b,c\}$. Thus, $f(y^{k})=b$ and we can repeat the argument above to show that $G$ is decisive over $b$. Thus, $G$ is decisive.
\end{proof}

\begin{lemma}[Contraction]
\label{group} 
If a group $G\subseteq I$ with $|G|\geq2$ is decisive, then it has a proper subset which is decisive.
\end{lemma}
\begin{proof}
We may again assume that $G=\{1, ..., k\}$ with $k\leq N$. Let $x\in X^{N}$ be a profile with $x_{1}=(a\succ ...\succ b)$, and for $2\leq i\leq k$, $x_{i}=(b\succ ... \succ a)$, and for all $k+1\leq j\leq N$, $x_{j}=(a\succ...\succ b)$. Then, by Lemma \ref{basic}, $f(x)\in\{a,b\}$. If $f(x)=b$ then we found $\{2, ..., k\}\subsetneq G$ which is decisive by Lemma \ref{field}. 

Assume $f(x)=a$ and we show that $\{1\}$ is decisive. Take $c\in A\setminus\{a,b\}$ and let $x^{1}\in X^{N}$ be the profile we obtained from $x\in X^{N}$ by changing $x_{1}$ with $x'_{1}=(a\succ b\succ ...\succ c)$. Then, $f(x^{1})=a$ as otherwise $f$ is MNP by individual $1$ at $x^{1}\in X^{N}$ via $x_{1}$. 
Let us start with $x^{1}\in X^{N}$ and change its $j'$th coordinate as $x'_{j}=(c\succ ...\succ b)$ for all $j\in \{k+1, ..., N\}$, one at a time. Let $x^{j}\in X^{N}$ be resulting profile after $x_{k+1}, ..., x_{j}$ are changed, and we set $x^{k}=x^{1}$. Then, $f(x^{j})=a$ for $j=k+1,..., N$, since 
\begin{itemize}
\item $f(x^{k})=a$,
\item $f(x^{j})\in\{a,b\}$ for all $j\in I\setminus G$, as otherwise individual $1$ can manipulate $f$ at $x^{j}\in X^{N}$ by reporting $b\in A$ at the top (recall that $G$ is decisive), and
\item if $f(x^{j})=a$, but $f(x^{j+1})\neq a$ for some $j\in \{k, ..., N-1\}$, then $f(x^{j+1})=b$, and $f$ is MNP by individual $j+1$ at $x^{j+1}$ via $x_{j+1}$.
\end{itemize} 
In particular, $f(x^{N})=a$. Let $y\in X^{N}$ be as $y=(x_{1}, x^{N}_{-1})$, i.e. the profile obtained from $x^{N}\in X^{N}$ after changing its first coordinate back to $x_{1}$. Then, $f(y)=a$ as otherwise $f$ is MNP at $y\in X^{N}$ by individual $1$ via $x'_{1}$. Start with $y\in X^{N}$ and for $2\leq i\leq k$ change $y_{i}=(b\succ ... \succ a)$ as $y'_{i}=(c\succ ... \succ a)$, one at a time. Let $y^{i}\in X^{N}$ be the profile obtained after changing $y_{2}, ..., y_{i}$, and we set $y^{1}=y$. Notice that $f(y^{i})= a$ for $i=2,..., k$, since
\begin{itemize}
\item $f(y^{1})=a$, and 
\item if $f(y^{i})=a$, but $f(y^{i+1})\neq a$ for some $i\in \{1, ..., k-1\}$, then $f$ is MNP by individual $i+1$ at $y^{i}\in X^{N}$ via $y'_{i+1}$.
\end{itemize}
In particular, $f(y^{k})=a$. Finally, start with $y^{k}\in X^{N}$ and for $k+1\leq j\leq N$ change its $j'$th coordinate as $y'_{j}=(c\succ ... \succ a)$, one at a time. Let $y^{j}\in X^{N}$ be the resulting profile after $y_{k+1}, ..., y_{j}$ are changed. Then, $f(y^{j})=a$ for $j=k+1,..., N$, since 
\begin{itemize}
\item $f(y^{k})=a$,
\item $f(y^{j})\in\{a,c\}$ for all $j\in \{k+1,..., N\}$ by Lemma \ref{basic}, and
\item if $f(y^{j})=a$, but $f(y^{j+1})\neq a$ for some $j\in \{k, ..., N-1\}$, then $f(y^{j+1})=c$, and $f$ is MNP by individual $j+1$ at $y^{j}\in X^{N}$ via $y'_{j+1}$.
\end{itemize} 
Thus, $f(y^{N})=a$, and $a\in A$ is ranked as the top by individual $1$, and as the bottom by everybody else. Then Lemma \ref{field} implies that $\{1\}$ is decisive.
\end{proof}

By UNM we know that $I$ is decisive. Then, repeated application of Lemma \ref{group} gives the result in Theorem \ref{GS}.
\section{Final Remarks}
Let us make a few comparisons. Lemma \ref{field} and \ref{group} are counterparts of the field extension and group contraction lemmas in \cite{sen86}, but with a small difference. Lemma \ref{field} states that if a group is decisive over an alternative at a particular profile, then it is globally decisive, whereas Sen's field extension lemma is not profile specific. Another difference between our proof and Sen's proof is Lemma \ref{basic}, which is a non trivial result that the former needs. 

The former difference can be attributed to the fact that the key axiom in Theorem \ref{GS}, strategy proofness, is a local (intra-profile) condition, while as already noted in \cite{fish87} most of the other key axioms such as monotonicity and Arrow's IIA are more global (inter-profile). The latter difference can be explained by different set ups used in stating impossibility results, i.e. social welfare function vs. social choice function. However, one can remove these differences by 
\begin{itemize}
\item Removing Lemma \ref{basic} and incorporating it in the proofs of the other lemmas wherever it is needed, and 
\item Replacing Lemma \ref{field} with the following weaker statement, proof of which is already embedded in the proof of Lemma \ref{field}: {\em If a group is decisive over $a\in A$, then it is decisive}. 
\end{itemize}
Such a change will make the two proofs parallel, but each of the resulting lemmas would have a longer proof. 

Since the differences are syntactical and can be removed, we believe that the above proof is very close to Sen's proof. Moreover, it is one of the shortest among the existing proofs of Theorem \ref{GS}, and it being missing is a surprise. One possible cause of this delay is the emergence of other approaches; in particular, the popular pivotal voter approach (see Sect. \ref{intro}). On the other hand, \cite{mou}, \cite{tay} and \cite{wall} apply variants of this approach to prove Theorem \ref{GS}. But the former two proofs use the Muller-Satterthwaite theorem (see Chap. 10.1 in \cite{mou}, Chap. 7.5 in \cite{tay}), while the latter proof consists of seven lemmas, one of which is stated without a proof (see Chap. 5.5 in \cite{wall}). As such, the current proof is more direct than any of these, and could well be `the missing proof' for an advocate of this approach.

\medskip\noindent {\bf Acknowledgements}: I am thankful to Mitsuyo Shirakawa for the guidance and to an anonymous referee for the remarks.

\end{document}